\documentclass[11pt, reqno]{amsart}
\usepackage{amsmath, amsthm, a4, latexsym, amssymb}

\def\@rmrk#1#2{\refstepcounter
    {#1}\@ifnextchar[{\@yrmrk{#1}{#2}}{\@xrmrk{#1}{#2}}}

\makeatletter\@addtoreset{equation}{section}\makeatother

 \newfont{\bfit}{cmbxti10 scaled 2000}
 \newfont{\biggi}{cmr12 scaled 2000}

 \newcommand{\N}{\mathbb{N}}

 \newcommand{\prob}{\mathbb{P}}

 \newcommand{\me}{\mathbb{E}}
 
 \renewcommand{\P}{\mathbb{P}}
 
 \newcommand{\skrim}{{\mathcal M}}

 \newcommand{\skrin}{{\mathcal N}}

 \newcommand{\heap}[2]{\genfrac{}{}{0pt}{}{#1}{#2}}

\def\1{{\mathchoice {1\mskip-4mu\mathrm l}      
{1\mskip-4mu\mathrm l}
{1\mskip-4.5mu\mathrm l} {1\mskip-5mu\mathrm l}}}

\newcommand{\eq}{\begin{equation}}
\newcommand{\en}{\end{equation}}

\newenvironment{Proof}
{\vskip0.1cm\noindent{\bf Proof. }{\hspace*{0.3cm}}}%
{\nopagebreak {\hspace*{\fill}\rule{2mm}{2mm}}\\ }

{\nopagebreak {\hspace*{\fill}\rule{2mm}{2mm}}\\ }

\renewcommand{\subsection}{\secdef \subsct\sbsect}
\newcommand{\subsct}[2][default]{\refstepcounter{subsection}
\vspace{0.15cm}
{\flushleft\bf \arabic{section}.\arabic{subsection}~\bf #1  }
\nopagebreak\nopagebreak}
\newcommand{\sbsect}[1]{\vspace{0.1cm}\noindent
{\bf #1}\vspace{0.1cm}}

\newtheorem{theorem}{Theorem}[section]
\newtheorem{lemma}[theorem]{Lemma}
\newtheorem{cor}[theorem]{Corollary}

\newtheoremstyle{thm}{1.5ex}{1.5ex}{\itshape\rmfamily}{}
{\bfseries\rmfamily}{}{2ex}{}

\newtheoremstyle{rem}{1.3ex}{1.3ex}{\rmfamily}{}
{\itshape\rmfamily}{}{1.5ex}{}
\theoremstyle{rem}

\refstepcounter{subsection}

\def\thebibliography#1{\section*{References}
  \list%
  {\arabic{enumi}.}
    {\settowidth\labelwidth{[#1]}\leftmargin\labelwidth
    \advance\leftmargin\labelsep
    \parsep0pt\itemsep0pt
    \usecounter{enumi}}
    \def\newblock{\hskip .11em plus .33em minus .07em}
    \sloppy                   
    \sfcode`\.=1000\relax}

\begin{document}
\title[LDP  for  Sample  Pooling ]
{\Large  Large  deviations, asymptotic bounds on  the  number  of  positive  individuals in  a  Bernoulli sample via  the  number of  positive  pool  samples drawn  on  the  bernoulli  sample }

\vspace{0.5cm}

\renewcommand{\thefootnote}{}

\renewcommand{\thefootnote}{1}

\maketitle
\thispagestyle{empty}
\vspace{-0.5cm}

\centerline{\sc By   Kwabena  Doku-Amponsah}
\centerline{University  of Ghana}

\renewcommand{\thefootnote}{}

\footnote{\textit {Affiliation:} University  of  Ghana, Department of  Statistics \& Actuarial Science}
\footnote{ \textit {Email:}kdoku-amponsah@ug.edu.gh}
\renewcommand{\thefootnote}{1}

\vspace{0.5cm}

\begin{quote}{\small }{\bf Abstract.}
In  this  paper   we  define  for a  Bernoulli  samples the \emph{ empirical infection  measure}, which counts  the  number of  positives (infections)  in  the Bernoulli sample  and  for  the \emph{ pool samples}  we  define  the  empirical pool  infection measure,  which  counts  the  number  of  positive  (infected) pool  samples.  For  this  empirical  measures  we  prove a  joint  large  deviation  principle  for Bernoulli  samples.  We  also found  an  asymptotic relationship  between  the \emph{ proportion  of  infected  individuals } with  respect  to the  samples  size, $n$  and  the \emph{ proportion of infected pool  samples}  with  respect  to  the  number  of  pool  samples, $k(n).$  All  rate  functions  are  expressed  in terms  of  relative  entropies.
\end{quote}\vspace{0.3cm}



\vspace{0.3cm}
\vspace{0.3cm}

\section{Introduction}

Sample pooling is a testing procedure mainly used in medical research to test several individuals at a time. In pool testing, samples from individuals are pooled together and tested for the presence of infectious diseases (specificity).For  instance, Noguchi Memorial Institute for  Medical Research (NMIMR) at  the  early  stages  of  the  COVID-19  pool up  samples at a time and test. This testing procedure allows the detection of positive samples with sufficient Positive Predictive Value (PPV) and detects negative samples with Sufficient Negative Predictive Value (NPV). If NMIMR pools individuals and test at once and the test results is negative, all the individuals who comprise of the pool sample is declared negative of COVID-19 resulting in huge cost saving because of the  inadequacy  of enough  testing instrument  or  equipment in Ghana.\\

 In  particular, fifteen samples may be tested together, using only the resources required for a single test. If a pool sample is negative, it can be inferred  that all individuals  were negative. If a pool sample comes back positive, then each sample needs to be tested individually to find out which individuals were positive. As samples are pooled together, ultimately fewer tests are run overall, meaning fewer testing supplies are needed, and results can be given to patients more quickly in most cases.   NMIMR, in the wake of COVID-19 in Ghana,  has  adopted  this  asymptomatic dignostic  testing  approach  in  order  to  come  with  large  number  of  testing  requirement  needs  of  the Ghana  as  we  have  crossed  the  35,000  confirm  cases  required  since  the  detection  of  the  first  positive  case  in  March  12,  2020. \\

Indeed, sampling pooling or  the asymptomatic  Dignostic  Testing of taking  up to fifteen (15) samples per pool  by the  NMIMR, has significantly  increased the   testing  capacity of  Ghana  given  the limited  resources such  as   equipment and test kits availability  in  the  country.\\

This method works well when there is a low prevalence of cases, meaning more negative results are expected.However, a  major  problem  will  arise  when  the  specificity  is  very  high  in  which  case  many  or  all  the  pool  samples  will  test  positive.  In  this  case  knowing  how  many  individuals  in  the  sample  are  asymptotically  infected  via  the  infected  pool  samples may  be key to the  management  of  the pandemic  situation.\\   

In  this  paper  we  find  an asymptotic function relationship  between   the  number  of positive individual cases  and  positive  pooled  samples as  the  number individual to  be tested  become  very  large.  To  be  specific,  we  define  the  empirical  proportion  measure  which  counts  the  number   of  infested  individual  in  the  sample  with  respect  to  the  sample  size  and  the  empirical pooled  proportion  measure  which  counts   the  number  of  infected  pooled  samples  with  respect  to  the  number  of  pool samples,   $k(n).$\\

 In  this sequel  we  define  the  two  main  objects  for   the  study: the  empirical  infection  measure  and  the  empirical pool infection  measure.  And  for  these  empirical  measures  we  prove  a  joint  large  deviation  principle with  speed  $n.$  From  this  large  deviation  principle  we  an  asymptotic  functional  relation  between  the  number  of  infected  individual ls  in  the  sample   and  the  number  infected  pools  samples  as   the  number  of  individuals increase. Further,  we  contract from our  main  large  deviation  principle  an  LDP  for  the  number  of  positive  cases  in  the  sample as  the  sample  size $n$  becomes  very  large.\\
 
 The  main  techniques  deplore  to  prove  our  LDPs  are  simlarly  to  the  ones  used  in  the  paper \cite{DM2010}  or  the  Ph.D Thesis  \cite{DA2006}: These  Gartner-Ellis  Theorem,see \cite{DZ1998},  Conbinatorics arguments  via  the  Method  of  types, and  method  of  Mixtures,  see \cite{Bi2004}.\\
 
 Note,  the  main different  between  our  main  LDP  for  the  empirical  measure  we  studied  in  this  paper   and  the  LDP  for  empirical measures  of   samples,  see  \cite{DZ1998}, is  that  while  the  earlier  result  is pooled  from  random  sample  and the  later  result   come  from  a  deterministic  space.  Also  large  deviation principle  for  sequences  of maxima and  minima    has  been  found  in  \cite{GM2014}.\\
 
 The  remaining  part of  the  article  is  organized  in  the  following  manner:  Section~\ref{Sec2}  contains the  main  background  of  the  article  and  the  statement  of  main  result. In Section~\ref{Sec3}  we  give  the  proofs  of  the  main  result.  The   derivation  of  Corollary~\ref{main0a} and  Conclustion  are   given in  Section~\ref{Sec4}.
 \pagebreak

\section{Statement  of Main Results }\label{Sinr}\label{Sec2}
\subsection{Empirical  Measures of  the Bernoulli Samples.}
Let  $X=(X_1, X_2, ...,X_{n})^{T}$  be  a  random  vector  of  independent  and  identically  distributed  Bernoulli  random  variables  each  with  success  probability  $\mu_n$.
  Let   $(N_1,N_2,N_3,...,N_{k(n)})$ be  a  random  sample  drawn uniformly  from  the set  of  integer  partitions  of  $n$ of  length $k(n)$.  Suppose the  components   of  $X$ are  grouped  into $Y_1=(X_1^{(1)},...,X_{N_{1}}^{(1)}),...,Y_{k(n)}(X_1^{(k)},...,X_{N_{k}}^{(k)})$   in  such  away  that  $X_i^{(r)}$  and $X_j^{(m)}$,  are  independent  for all  $r\not=m,$ and $i\not=j$,  where  $i,j=1,2,3,...$.   We shall  call  $(X_1^{(1)},...,X_{N_{1}}^{(1)}),...,(X_1^{(k)},...,X_{N_{k}}^{(k)})$  pool  samples taken  from  the  component.  of  $X.$ 
 We  shall  study  the  pool  samples  under  the    the  law  of  the  pool  samples:
 
  $$ \P_n\Big\{Y_1,...,Y_{k(n)}\Big\}:=\P\Big\{Y_1,...,Y_k(n)\,  |N_{1}\not =0,..., N_{k(n)}\not =0, N=n\Big\},$$ 
  
  where  
  $$N_1+N_2+N_3+...+N_{k(n)}=N.$$

For any  sample  $X$   we  define two measures, the
\emph{empirical infection  measure}, ~ $P_1^{X}\in\skrim(\{0,1\})$,~by
$$P_1^{X}\big(a\big ):=\frac{1}{n}\sum_{i\in [n]}\delta_{X_i}\big(a\big)$$
and the  empirical  pooled  infection  measure
$P_2^{X}\in\skrim(\N\times \skrin(\{0,1\})),$ by
$$P_2^{X,N}\big(m,\ell\big):=\frac{1}{n\beta_n}\sum_{j\in[k(n)]}\delta_{(N_j,L_j )}\big(m,\ell\big),$$

where  $L_j=\Big(\frac{1}{N_j}\sum_{i=1}^{N_j}\delta_{X_i^{(j)}}(x)\,,x\in\{0,1\}\Big) $    and  $\beta_n=k(n)/n$. Note  that the  total mass of  the  empirical infection  measure  is  $\1$   and  the  empirical  pool  infection  measure  is  also $\1$. Observe  also  that  $$\,P_1^{X}(x)=\beta_n
 \sum_{(m,l)\in\N\times\skrin(\{0,1\})} m\ell(x)\,P_2^{X}(m,\ell).  $$

 By  $\skrim(\{0,1\})$  we  denote  the  space  of  probability  vectors  on  the  polish  space   $\{0,1\}$  equipped  with  the weak  topology.By  $\skrin(\{0,1\})$   we  denote  space  probability  measures  on  $\{0,1\}$ and  we  assume by  convention for  any pair $(m,\ell)$ we  have  $\ell=(0,0)$   iff   $m=0.$  The  first theorem  in  this  Section, Theorem~\ref{main1a}, is  the  joint  LDP  for  the  empirical infection  measure  and  the  empirical  pool infection  measure.\\

 \subsection{Main Results.}
     We  assume  henceforth,  $k(n)\to\infty,$ $\beta_n:=k(n)/n\to \beta\in(0,1]$  as  $n\to\infty$  and  $n\mu_n/k(n)\to q.$
 Observe,by   the  fore-mentioned  assumptions,we  have $$\beta\Big(q(0)+q(1)\Big)=1.$$ 
 
 We  define  the  rate  function  $J_{\beta}:\skrim(\{0,1\})\times\skrim(\N\times\skrin(\{0,1\}))\to(0,\infty]$
  by
 \begin{equation}
 \begin{aligned}
 J_{\beta}(\omega,\,\pi)= \left\{\begin{array}{ll}  H\Big(\omega/\beta\Big\| \,q\,\Big)+ H\Big(\pi\|\Phi_\beta^{\omega}\Big),  & \mbox{if\qquad  $\langle \pi\rangle = \omega/\beta.$  }\\
 \infty & \mbox{otherwise,}		
 \end{array}\right.
 \end{aligned}
 \end{equation} 
 
 where  $$\langle \pi\rangle(x)=\sum_{(m,\ell)\in \N\times \skrin(\{0,1\})}m\ell(x)\pi(m,\ell)$$
 
 and  the  trivariate probability  distribution

 \begin{equation}
 \Phi_{\beta}^{\omega}(m,\ell)=\frac{1}{(1-e^{1/\beta})}\Big[\frac{[\omega(0)/\beta]^{[m\ell(0)]}e^{-[\omega(0)/\beta]}}{[m\ell(0)]!}\Big]\Big[\frac{[\omega(1)/\beta]^{[m\ell(1)]}e^{-[\omega(1)/\beta]}}{[m\ell(1)]!}\Big],
 \end{equation} 
 where   $m\in\N$  and  $\ell\in\skrin(\{0,1\}).$
 The  first  theorem, Theorem~\ref{main1a}  is our main  results  the  joint  LDP  for  the  empirical  infection  measure  and  the  empirical  poll infection  measure.

\begin{theorem}\label{main1a}
	Suppose  $X=(X_1, X_2, ...,X_{n})^{T}$  is a  random  vector  of  independent  and  identically  distributed  Bernoulli  random  variables  each  with  success  probability  $\mu_n$ satisfying  $n\mu_n(x)/k(n)\to q(x),$ for $x\in\{0,1\}.$
	 Let  $(X_1^{(1)},...,X_{N_{1}}^{(1)}),...,(X_1^{(k)},...,X_{N_{k}}^{(k)})$   be  pool samples  drawn  from components   of  $X$.
		Then, as $n\rightarrow\infty,$  the  pair
$(P_1^{X},P_2^{X,N})$  satisfies a  large  deviation  principle  in   the  space 
$\skrim(\{0,1\})\times\skrim(\N\times\skrin(\{0,1\}))$
with good rate function  $\beta J_{\beta}(\omega,\,\pi)$.

\end{theorem}

  From Theorem~\ref{main1a}   we  obtain  the  Corollary~\ref{main0a}  which  gives   us  the  asymptotic  relationship  between  the  number  of  infected  individual  in  the  sample  and  the  number  of  infected pool  samples.\\
  
\begin{cor}\label{main0a}
	Suppose  $X=(X_1, X_2, ...,X_{n})^{T}$  is a  random  vector  of  independent  and  identically  distributed  Bernoulli  random  variables  each  with  success  probability  $\mu_n$ satisfying  $n\mu_n(x)/k(n)\to q(x),$ for $x\in\{0,1\}.$	Then,   the  proportion  of  positive individuals, $I$,  satisfies

\begin{equation}\label{cor1}\begin{aligned}
\lim_{n\to\infty}\frac{1}{n}\log\P_n\Big\{I=t\Big|S=\sigma\Big\}	=-\Big[(1-t)\log\Big(\frac{(1-t)}{(1-\beta q(1))}\Big)+  t\log\frac{t}{\beta q(1)}\Big]
	\end{aligned}
	\end{equation}
	where $$t=-\beta\log\Big[1-(1-e^{-1/\beta})\sigma\Big].$$

	\end{cor}

The  next  LDP,  Theorem~\ref{main1b} will  play  a  vital  role  in  the  proof  of  Theorem~\ref{main1a}.
	
\begin{theorem}\label{main1b}
	Suppose  $X=(X_1, X_2, ...,X_{n})^{T}$  be  a  random  vector  of  independent  and  identically  distributed  Bernoulli  random  variables  each  with  success  probability  $n\mu_n/k(n)\to q(x)$.
	Let  $(X_1^{(1)},...,X_{N_{1}}^{(1)}),...,(X_1^{(k)},...,X_{N_{k}}^{(k)})$   be  pool samples  drawn  from components   of  $X$ conditional  on  the  event  $ \Big\{ P_1^{X}=\omega \Big\}$. 
	Then, as $n\rightarrow\infty,$   
	$	P_2^{X,N}$  satisfies a  large  deviation  principle  in   the  space 
	$\skrim(\N\times\skrin(\{0,1\}))$
	with good rate function  $\beta J_{\omega}^{\beta}(\pi),$ where
	
	\begin{equation}
	\begin{aligned}
J_{\omega}^{\beta}(\pi)=\left\{\begin{array}{ll}  H\Big(	\pi\,\|\Phi_{\beta}^\omega\Big),  & \mbox{if\qquad  $\langle \pi\rangle = \omega/\beta.$  }\\
	\infty & \mbox{otherwise.}		
	\end{array}\right.
	\end{aligned}
	\end{equation} 
	
	\end{theorem}

The  next  LDP,  Theorem~\ref{main1c} will   also  play  a  vital  role  in  the  proof  of  Theorem~\ref{main1a}.

\begin{theorem}\label{main1c}

Suppose  $X=(X_1, X_2, ...,X_{n})^{T}$  be  a  random  vector  of  independent  and  identically  distributed  Bernoulli  random  variables  each  with  success  probability  $n\mu_n(x)/k(n)\to q(x)$.
Let  $(X_1^{(1)},...,X_{N_{1}}^{(1)}),...,(X_1^{(k)},...,X_{N_{k}}^{(k)})$   be  pool samples  drawn  from components   of  $X$. 
	Then, as $n\rightarrow\infty,$  
	$	P_1^{X}$  satisfies a  large  deviation  principle  in   the  space 
	$\skrim(\{0,1\}\})$
	with good rate function $\beta I^{\beta}(\omega),$  while
	
	\begin{equation}
	I^{\beta}(\pi)=  H\Big(	\omega/\beta\,\|\,q\Big).
	\end{equation} 
	
\end{theorem}

\section{ Proof  of  Theorem~\ref{main1b}  and  Theorem~\ref{main1c}}\label{proofmain1bc}\label{Sec3}

\subsection{Proof  of   Theorem~\ref{main1b}  by Method of  Types }\label{proofmain}

We  begin  the  proof  of   Theorem~\ref{main1b}  by  first  defining two  main  sets: 
	 $$\skrim_n(\{0,1\}):=\Big\{\omega_n\in\skrim(\{0,1\}):n\omega_n(0),n\omega_n(1)\in\N\Big\}$$
	 and  
	 
	 $$\skrim_n(\N\times \skrin(\{0,1\})):=\Big\{\pi_n\in\skrim(\N\times \skrin(\{0,1\})):n\pi_n(m,\ell)\in\N ,\mbox{  for  all $(m,\ell)\in\N\times\skrin(\{0,1\})$}\Big\}.$$
	 
	 The  nest  Lemma~\ref{Types}  is  a  major  step  in  the  proof  of  Theorem~\ref{main1b}.
	
\begin{lemma}\label{Types}
Suppose  $X=(X_1, X_2, ...,X_{n})^{T}$  be  a  random  vector  of  independent  and  identically  distributed  Bernoulli  random  variables  each  with  success  probability  $\mu_n$.
Let  $(X_1^{(1)},...,X_{N_{1}}^{(1)}),...,(X_1^{(k)},...,X_{N_{k}}^{(k)})$   be  pool samples  drawn  from components   of  $X$.  Suppose  $(\omega_n,\,\pi_n)\to(\omega,\pi)\in\skrim(\{0,1\})\times\skrim(\N\times\skrin(\{0,1\})).$ Then,  we  have 	
$$e^{-n\beta_nH\Big(\pi_N\,\Big|\,\Phi_{\beta_n}^{\omega_n}\Big)+n\eta_1(n)}\le\P_n\Big\{P_2^{X,N}=\pi_n  \Big | P_1^{X}=\omega_n\Big\}\le e^{-n\beta_nH\Big(\pi_n\,\Big|\,\Phi_{\beta_n}^{\omega_n}\Big)+n\eta_2(n)},$$
$$\lim_{n\to\infty}\eta_1(n)=\lim_{n\to\infty}\eta_2(n)=0.$$
\end{lemma}

\begin{Proof}
We  begin  by  observing that  conditional  on  the  event $\Big\{P_1^{X}=\omega_n\Big\}$  the   law  of  $P_2^{X,N}=\pi_n$  is  giving  by  $$\begin{aligned}
&\P_n\Big\{P_2^{X,N}=\pi_n  \Big | P_1^{X}=\omega_n\Big\}\\
&=\Big(\heap{n}{n\beta_n\pi_n(m,\ell),\,(m,\ell)\in\N\times\skrin(\{0,1\})}\Big)\times\prod_{x\in\{0,1\}}\Big( \heap{n\omega_n(x)}{m_1\ell_1(x),m_2\ell_2(x),...,m_{k(n)}\ell_{k(n)}(x)}\Big)\Big(\frac{1}{n}\Big)^{n\omega_n(x)}\\
&\times\prod_{(m,\ell)\in\N\times\skrin(\{0,1\})} \Big[1-\prod_{x\in\{0,1\}}\Big(1-\frac{1}{k(n)}\Big)^{n\omega(x)}\Big]^{-n\beta_n\pi_n(m,\ell)}
\end{aligned}$$ 

We recall for  any  $n\in\N$, the  refined  Stirling's  Formaula  as  $$  n^n e^{-n}\le n!\le (2\pi n)^{-1/2}n^n e^{-n+1/(12n)}.$$

Now  using  the  refined Stirling's  Formula  we  have
\begin{equation}\label{Lemma4}
\begin{aligned}&-n\sum_{(m,\ell)\in\N\times\skrin(\{0,1\})}\beta_n\pi_n(m,\ell)\log\beta_n\pi_n(m,\ell)-\sum_{(m,\ell)\in \N\times \skrin(\{0,1\})}\frac{1}{12n\beta_n\pi_n(m,\ell)}\\
&-n\sum_{(m,\ell)\in\N\times\skrin(\{0,1\})}\beta_n\pi_n(m,\ell)\log\Big(1-e^{-1/\beta_n}\Big)-o(1)\\
&\le\log\Big(\heap{n}{n\beta_n\pi_n(m,\ell),\,(m,\ell)\in\N\times\skrin(\{0,1\})}\Big)\prod_{(m,\ell)\in\N\times\skrin(\{0,1\})} \Big[1-\prod_{x\in\{0,1\}}\Big(1-\frac{1}{k(n)}\Big)^{n\omega(x)}\Big]^{-n\beta_n\pi_n(m,\ell)} \\
&\le-n\sum_{(m,\ell)\in\N\times\skrin(\{0,1\})}\beta_n\pi_n(m,\ell)\log\beta_n\pi_n(m,\ell)+\frac{1}{12n} +\frac{1}{2}\sum_{(m,\ell)\in\N\times\skrin(\{0,1\})}\log \,2\pi n\beta_n\pi_n(m,\ell)\\
&-\sum_{(m,\ell)\in\skrin(\{0,1\})}\frac{1}{12n\beta_n\pi_n(m,\ell)}
-n\sum_{(m,\ell)\in\N\times\skrin(\{0,1\})}\beta_n\pi_n(m,\ell)
\log\Big(1-e^{-1/\beta_n}\Big)+o(1)
\end{aligned} 
\end{equation}

Also  we  have
\begin{equation}\label{Lemma5}
\begin{aligned}
&\sum_{x\in\{0,1\}}n\omega_n(x)\log\omega_n(x)-n\sum_{x\in\{0,1\}}\omega_n(x)-n\beta_n\sum_{m\in\N}\sum_{\ell\skrin(\{0,1\})}\log\ell(x)!\pi_n(m,\ell)\\
&\le \log\prod_{x\in\{0,1\}}\Big(\heap{n\omega_n(x)}{m_1\ell_1(x),m_2\ell_2(x),...,m_{k(n)}\ell_{k(n)}(x)}\Big)\Big(\frac{1}{n}\Big)^{n\omega_n(x)}\\
&\le\sum_{x\in\{0,1\}}n\omega_n(x)\log\omega_n(x)-n\sum_{x\in\{0,1\}}\omega_n(x)-n\beta_n\sum_{m\in\N}\sum_{\ell\in \skrin(\{0,1\})}\log\ell(x)!\pi_n(m,\ell)+n\beta_n\sum_{x\in\{0,1\}}\frac{1}{n\omega(x)}.
\end{aligned}
\end{equation}
\end{Proof}

Combining  \ref{Lemma4}  and  \ref{Lemma5}  and  taking  $$\eta_1(n)=-\sum_{(m,\ell)\in \N\times \skrin(\{0,1\})}\frac{1}{12n\beta_n\pi_n(m,\ell)}+o(1)\qquad\mbox{and}\qquad  \,\eta_2(n)=\sum_{x\in\{0,1\}}\frac{1}{n\omega(x)}+o(1)$$  we  have  

$$\begin{aligned}&-n\beta_n H\Big(\pi_n\,\Big|\,\Phi_{\beta_n}^{\omega_n}\Big)+n\eta_1(n)
\le\log\P_n\Big\{P_2^{X,N}=\pi_n  \Big | P_1^{X}=\omega_n\Big\}
\le -n\beta_nH\Big(\pi_n\,\Big|\,\Phi_{\beta_n}^{\omega_n}\Big)+n\eta_2(n),
\end{aligned}$$

where   \begin{equation}
\Phi_{\beta}^{\omega}(m,\ell)=\frac{1}{(1-e^{-1/\beta})}\Big[\frac{[\omega(0)/\beta]^{[m\ell(0)]}e^{-[\omega(0)/\beta]}}{[m\ell(0)]!}\Big]\Big[\frac{[\omega(1)/\beta]^{[m\ell(1)]}e^{-[\omega(1)/\beta]}}{[m\ell(1)]!}\Big],
\end{equation} 
where   $m\in\N$  and  $\ell\in\skrin(\{0,1\})$   which  ends  prove  of  the  Lemma~\ref{Types}.

\begin{lemma}\label{Com7}  Suppose  $(\omega_n,\,\pi_n)$  converges to  $(\omega,\,\pi)$ in  the  space  $\skrim(\{0,1\})\times\skrim(\N\times\skrin(\{0,1\})).$  Then,   
$$\lim_{n\to\infty}\Big|H\Big(\pi_n\,\Big|\,\Phi_{\beta_n}^{\omega_n}\Big)-H\Big(\pi\,\Big|\,\Phi_{\beta_n}^{\omega}\Big)\Big|=0.$$
\end{lemma}
\begin{Proof}
By  triangle  inequality  we  have  $$\Big|H\Big(\pi_n\,\Big|\,\Phi_{\beta_n}^{\omega_n}\Big)-H\Big(\pi\,\Big|\,\Phi_{\beta}^{\omega}\Big)\Big|\le\Big|H\Big(\pi_n\Big|\Phi_{\beta_n}^{\omega_n)}
\Big)-H\Big(\pi_n\Big|\Phi_{\beta}^{\omega}\Big)\Big|+\Big|H\Big(\pi_n\Big|\Phi_{\beta}^{\omega}
\Big)-H\Big(\pi\Big|\Phi_{\beta}^{\omega}
\Big)\Big|$$

Now, $\displaystyle \Big|H\Big(\pi_n\Big|\Phi_{\beta}^{\omega}
\Big)-H\Big(\pi\Big|\Phi_{\beta}^{\omega}
\Big)\Big|=0$  by  the  continuity  relative entropy  and   $$\begin{aligned}&\Big|H\Big(\pi_n\Big|\Phi_{\beta_n}^{\omega_n)}
\Big)-H\Big(\pi_n\Big|\Phi_{\beta}^{\omega}\Big)\Big|\\
&=\sum_{x\in\{0,1\}}\omega_n(x)\log \omega_n(x)-\sum_{x\in\{0,1\}}\omega(x)\log\omega(x)-\sum_{x\in\{0,1\}}\omega_n(x)+ \sum_{x\in\{0,1\}}\omega(x)\\
&+\beta_n\sum_{(m,\ell)\in\N\times\skrin(\{0,1\})}\sum_{x\in\{0,1\}}\log m\ell(x)![\pi_n(m,\ell)-\pi(m,\ell)]+ \log\Big(1-e^{-1/\beta_n}\Big)-
\log\Big(1-e^{-1/\beta}\Big)
\end{aligned}$$

$$\begin{aligned}\Big|H\Big(\pi_n\Big|\Phi_{\beta_n}^{\omega_n}
\Big)-H\Big(\pi_n\Big|\Phi_{\beta}^{\omega}\Big)\Big|&\le\sum_{x\in\{0,1\}}\omega_n(x)\log \omega_n(x)-\sum_{x\in\{0,1\}}\omega(x)\log\omega(x)
+\frac{1}{n\beta_n}\log n![k(n)-k(n))]\\
&+\log\Big(1-e^{-1/\beta_n}\Big)-
\log\Big(1-e^{-1/\beta}\Big)
\end{aligned}$$

Taking  the  limit  as  $n\to\infty$   we  have  $\displaystyle \Big|H\Big(\pi_n\Big|\Phi_{\beta_n}^{\omega_n)}
\Big)-H\Big(\pi\Big|\Phi_{\beta}^{\omega}\Big)\Big|\to 0,$  which  completes  the  proof  of  Lemma~\ref{Com7}.

\end{Proof}

Now,  using   similar  Integer  partitions  arguments as in \cite{DM2010}  applied  to  the  method of  types, see \cite[Proof  of  Theorem~1.1.10]{DZ1998}  and   Lemma~\ref{Com7} above, we  have  the LDP for  the  empirical pool infection measure   conditional   on  the  empirical infection  measure with  rate  function  $\beta H\Big(\pi\Big|\Phi_{\beta}^{\omega}\Big)$  which ends the   proof  in   Theorem~\ref{main1b}.

\subsection{Proof  of  Theorem~\ref{main1c}  by  Gartner-Ellis  Theorem}
The  next Lemma~\ref{Lemma2} will be  vital  for  using  The  Garner-Ellis   Theorem in  the  Proof  of Theorem~\ref{main1c}.

	\begin{lemma}\label{Lemma2}
Suppose  $X=(X_1, X_2, ...,X_{n})^{T}$  is  a  random  vector  of  independent  and  identically  distributed  Bernoulli  random  variables  each  with  success  probability  $n\mu_n(x)/k(n)\to q(x)$.
Let  $(X_1^{(1)},...,X_{N_{1}}^{(1)}),...,(X_1^{(k)},...,X_{N_{k}}^{(k)})$   be  pool samples  drawn  from components   of  $X$. 
	Then,
	
	$$\lim_{\lambda\to\infty}\frac{1}{n}\log\me\Big\{e^{n\langle g, \,  P_1^{X}\rangle }\Big\}=-\sum_{x=0}^{1}\Big( \Big[1-e^{g(x)}\Big]\, \beta q(x)\Big)
	=-\Big\langle 1-e^{g},\, \beta 	q\Big\rangle.$$
\end{lemma}	
\begin{Proof}
	$$\begin{aligned}
	\me\Big\{e^{n\langle g, \,  P_1^{X}\rangle }\Big\}&=\me\Big\{\prod_{i=1}^{n}e^{g(X_i)}\Big\}\\
	&=\prod_{i=1}^{n}\me\big(e^{g(X_i)}\big)\\
	&=\prod_{x\in\{0,1\}}\Big(1-\mu_n(x)+e^{g(x)}\mu_n(x)\Big)^n\\
	&=\prod_{x\in\{0,1\}}e^{-(1-e^{g(x)})q(x)k(n)+o(n)}
	\end{aligned}$$
	
	Talking limit of  the  normalized  logarithm  we  have  
	
	$$\lim_{\lambda\to\infty}\frac{1}{n}\log\me\Big\{e^{n\langle g, \,  P_1^{X}\rangle }\Big\}=-\sum_{x=0}^{1}\Big( \Big[1-e^{g(x)}\Big]\, \beta q(x)\Big)
	.$$
	
	which  ends  the  proof  of  the  Lemma.
	
\end{Proof}

Now, using  Gartner-Ellis Theorem,the  probability measure  $P_1^{X} $   obeys  an  LDP  with  speed  $n$   and  rate  function  $$I(\omega)=\sup_{g}\Big\{\Big\langle g,\omega\rangle +\Big\langle (1-e^{g})\,,\, q\beta\Big\rangle \Big\}.$$

By  solving  th  variational  problem   we   have the  relative  entropy   
$$I(\omega)=\beta H\Big(\omega/\beta\Big\|\,q \Big)$$    which  proves  Theorem~\ref{main1c}.

\section{ Proof of  Theorem~\ref{main1a}  by  Method  of  Mixtures}\label{Sec4}

For  each $n\in\N$ we define
$$\begin{aligned}
\skrim_{n}(\{0,1\}) & := \Big\{ \omega_n\in \skrim(\{0,1\}) \, : \, n\omega_n(x) \in \N \mbox{ for all } x\in \{0,1\}\Big\},\\
 \skrim_{n }(\N\times \skrin(\{0,1\})) & := \Big\{ n\beta_n\pi_n\in
\skrim(\N\times \skrin(\{0,1\})) \, : \, 
n\beta_n\,\pi_n(x,y) \in \N,\,  \mbox{ for all } \, x,y\in \{0,1\}
\Big\}\, .
\end{aligned}$$

We denote by
$\Theta_{n}:=\skrim_{n }(\{0,1\})$
and
$\Theta:=\skrim(\{0,1\})$.
With
$$\begin{aligned}
P_{n}(\omega_n) & := \prob\Big\{P_1^X=\omega_n \, \big\},
\end{aligned}$$
$$P_{\omega_{n}}(\pi_n):=\P\Big\{P_2^X=\pi_{n} \big|P_1^X=\omega_n\Big\}$$
the joint distribution of $P_1^{X}$ and $P_2^{X}$ is
the mixture of $P_{n}$ with
$P_{\omega_n},$ 
as follows: 
\begin{equation}\label{randomg.mixture}
d\tilde{P_n}( \omega_{n}, \pi_{n}):= dP_{	\omega_n }(\pi_{n})\, dP_n( \omega_{n}).\,
\end{equation}

Biggins~\cite[Theorem~5(b)]{Bi2004}  gives criteria for the validity of
large deviation principles for the mixtures and for the goodness of
the rate function if individual large deviation principles are
known. The following three lemmas ensure validity of these
conditions.

Observe   that the  family of
measures $(P_{n} \colon n\in\N)$  is  exponentially tight on
$\Theta.$

\begin{lemma}[] \label{Com4}

	The  family	measures $(\tilde{P}_{n} \colon n\in \N$  is  exponentially tight on
	$\Theta\times\skrim(\{0,1\}\times \skrin(\{0,1\})).$

\end{lemma}
\begin{Proof}
	
 Note,  $P_{\omega_n}$   is  a  probability   distribution on  space  of   positive  finite  measures  and so  using  similar  argument as  in  the  proof  of  \cite[Lemma~4.3]{SAD2020},   we  can  conclude   $P_{	\omega_n}$   is  exponentially  tight.  Moreover,  $P_n$  is  a  probability distribution on  the space  of  probability vectors on $\{0,1\}$   so  by  the  Chebychev's  inequality and  the Prokhov's  Theorem,  we can  conclude  $P_{n}$ is  exponentially  tight.    Hence,  as  $\tilde{P}_{n} $  is  mixture  of  two  exponentially tight probability   distributions ( $P_{\omega_n}$ and $P_{n}$ ),   we  can  conclude that   the  sequence  of  measures   $(\tilde P_{n} \colon n\in\N)$  is  exponentially  tight  on	$\Theta\times\skrim(\{0,1\}\times \skrin(\{0,1\})).$   See, example~\cite{DM2010}.
\end{Proof}

Define the function
$I\colon{\Theta}\times\skrim(\{0,1\}\times \skrin(\{0,1\}))\rightarrow[0,\infty],$  by

	\begin{equation}
\begin{aligned}
J_{\omega}^{\beta}(\pi)= \left\{\begin{array}{ll}  H\Big(	\pi\,\|\Phi_{\beta}^\omega\Big),  & \mbox{if\qquad  $\langle \pi\rangle = \omega/\beta.$  }\\
\infty & \mbox{otherwise.}		
\end{array}\right.
\end{aligned}
\end{equation} 

  and  recall  from  Theorem~\ref{main1c}  that  
   \begin{equation}
   I^{\beta}(\pi)=  H\Big(\omega/\beta\,\|\,q\Big).
   \end{equation}

\begin{lemma}[]\label{Com5}
\begin{itemize}
	
\item[(i)]	$J_{\omega}^{\beta}$ is lower semi-continuous.
\item  [(ii)] $J_\beta$  is  lower  semi-continuous.
\end{itemize}
\end{lemma}
\begin{proof}
Observe  that     $ H\Big(\omega\,\Big |\beta q\Big)$ is relative  entropy functions  by  definition. $J_\beta$  is  the  sum  of  two relative  entropy  functions  since  $J_\omega^{\beta}$  is  a  function  of relative 
     entropy.  We  conclude  that  both   $J_\omega^{\beta}$ and  $J_{\beta}$  are  lower  semi-continuous.
 \end{proof}

By ~\cite[Theorem~5(b)]{Bi2004}, the two previous lemmas and the
large deviation principles we have proved,
Theorem~\ref{main1b}  and  Theorem~\ref{main1c} ensure
that under $\tilde{P}_{n}$    and   $P_{n}$ the random variables $(\omega_{n}, \pi_{n})$    satisfy a large deviation principle on
$\skrim(\{0,1\}) \times\skrim(\N\times\skrin(\{0,1\}))$   with good rate function  $\tilde{J}$  respectively,  which  ends  the  proof of  Theorem~\ref{main1a}.

\section{Proof  of  Corollary~\ref{main0a}  and  Conclusion }\label{Sec4}

\subsection{Proof  of  Corollary~\ref{main0a}.}
 The  proof  of  the  Corollary   is  obtain  from  theorem~\ref{main1a}   by  applying  the  contraction  Principle, see \cite{DZ1998},  to  the  linear  mapping  $(\omega,\pi)\to  \omega(1).$ 
By  theorem~\ref{main1a}  $(P_1^X, P_2^{X,N})$ obeys   a  large  deviation  principle   with  speed $n$  and  rate  function $J(\omega,\pi)$.Applying  the  contraction  principle   to   the  linear  mapping  above  we  have  that  $I=P_1^{X}(1) $   obeys  an  LDP  with  speed  $n$  and  rate  function 

\begin{equation}\label{cor2}
\phi_{\sigma}(t)=\inf\Big\{J(\omega,\pi): \omega(1)=t, 1-\pi_2((0,1))=\sigma\Big\},
\end{equation}

where  $$\pi_2((0,1))=\sum_{m=1}\Phi_{\beta}^{\omega}(m,(0,1))=\frac{1}{(1-e^{-1/\beta})}e^{-\omega(1)/\beta}(1-e^{-(1-\omega(1))/\beta}).$$

Solving  the  optimization  problem  in  \ref{cor2}  we    have  the  rate  function \ref{cor1}  which  completes  the  proof  of  Corollary~\ref{main1a}.

\subsection{Conclusion.}In this  article we  have  found  an  asymptotic  bounds  for  the  number  of  infected(positive) individuals in  a  Bernoulli  sample  via  the  number  of  infected  pool  samples.  From  the  asymptotic relationship in  Corollary~\ref{main0a} we  can  infer  the  following: Note  $0\leq \sigma\leq 1$  and  

\begin{itemize}
	\item	If  $\sigma=0$  we shall  say  there  are  no  infected pool  samples  and  hence we  no  infected individuals  in  the  sample  $X,$ with  a positive  probability.
	
	\item If  $\sigma=1$  we  shall  say nearly  every pool sample   formed  from   $X$  is  infected  and  hence  we  have  that  every  individual in  the  sample  $X$ is  infected  with  a positive probability.
	
	\item If  $0<\sigma<1$   we  shall  say  that some   of  the  pool  samples form from  $X$  are  infected  and  hence  some individuals  in  the  sample  $X$   are  infected with  a  positive  probability.  
\end{itemize}

In  fact  we   shall  say  that  $nI$  individuals   of  the  sample  are  infected,  with a positive  probability.  Particularly, we  have  indirectly established  that  $nI\, \ge\, n\beta\sigma,$  with  probability $1.$


\end{document}